\documentclass{amsart}
\usepackage{amsmath,amssymb,amscd,amsthm,verbatim,alltt,amsfonts,array}
\usepackage[english]{babel}
\usepackage{latexsym}
\usepackage{amssymb}
\usepackage{euscript}
\usepackage{graphicx}
\usepackage[all]{xy}

\newtheorem{theorem}{Theorem}
\theoremstyle{plain}

\newtheorem{example}{Example}

\newtheorem{remark}{Remark}

\numberwithin{equation}{section}

\def\sing{\textsc{Singular}}

\begin{document}

\title{Algorithms in the classical N\'eron Desingularization}

\author{ Asma Khalid, Adrian Popescu and Dorin Popescu}

\date{}

\pagestyle{myheadings}
\markboth{Author(s) }{Title }

\maketitle

\begin{abstract}
   We give algorithms to construct the N\'eron Desingularization and the easy case from \cite{KK} of the General  N\'eron Desingularization.
\end{abstract}

\begin{quotation}


\end{quotation}


\thispagestyle{empty}

\section*{Introduction}

Let $A\subset A'$ be an unramified extension of discrete valuation rings (shortly DVR), that is the generator $x$ of the maximal ideal of $A$ generates also the maximal ideal of $A'$. Suppose that the induced field extensions $A/(x)\to A'/xA'$, Fr$(A)\to$ Fr$(A')$ are
separable, in other words the inclusion $A\to A'$ is regular. N\'eron \cite{N} (see here Theorem \ref{n}) proved that every sub-$A$-algebra $B\subset A'$ of finite type can be embedded in a regular  subring $D\subset A'$, essentially of finite type over $A$. Moreover, $D$ could be chosen contained in Fr$(B)$, that is $D$ could be seen as a  desingularization of $B$ made with respect of the inclusion $B\to A'$.

By Jacobian Criterion we see that $D$ is a localization of a smooth sub-$A$-algebra $C\subset A'$ containing $B$. Suppose that $B=A[Y]/(f)$, $Y=(Y_1,\ldots,Y_m)$, $f=(f_1,\ldots,f_p)$ and the inclusion $v:B\to A'$ is given by $Y\to y\in A'^m$. Replacing $B$ by $C$ means in fact to substitute the system of polynomial equations $f$ by a system of polynomials in more variables where it is possible to apply the Implicit Function Theorem. If $A'$ is Henselian then this could be very helpful. One example is the applications of N\'eron  Desingularization  to the  so-called the Artin approximation property of algebraic power series rings with coefficients in an excellent Henselian DVR (see \cite{A}).
Actually for this aim it is not necessary the  injectivity of $v$ and $C\subset A'$. Ploski
\cite{Pl} proved that if $A=\mathbb{C}\{x\}$, $x=(x_1,\ldots,x_n)$, $B=\mathbb{C}\{x,Y\}/(f)$ for some complex convergent power series $f$ in $x,Y$   then a $\mathbb{C}$-morphism $v:B\to  \mathbb{C}[[ x ]]$ factors through an $A$-algebra of type $\mathbb{C}\{x,Z\}$ for some new variables $Z$.

Ploski's result gave the idea of the so-called the  General N\'eron Desingularization, (see \cite{P0}, \cite{P1},   \cite{S}) which says that   for special (that is regular) morphisms $A\to A'$ of Noetherian rings any $A $-morphism $v:B\to A'$ with  $B$ a  $A$-algebra of finite type, factors through a smooth $A$-algebra $C$, that is $v$ is a composite $A$-morphism $B\to C\to A'$. This desingularization is  constructive when $\dim  A=\dim A'=1$ (see \cite{AP}, \cite{PP}, \cite{KPP}).

In \cite{KK} an easy proof of the   General N\'eron Desingularization  is given in the case when $\dim A=0$, $\dim A'=1$ and $A$, $A'$ have the same residue field. Here
it is our goal  to  provide some algorithms in this direction. We start presenting an algorithm to construct the N\'eron
Desingularization in the case when $A$, $A'$ are DVR. In Section 1 we recall   N\'eron's proof in the idea presented by Artin \cite[Theorem 4.5]{A}  in a special case (see also \cite[Theorem 5.1, page 683]{P'}, \cite[pages 171-172]{PP0}). The present proof  is taken from \cite{P2} and we give it here for the sake of completeness.
 In Sections 2, 3 we present shortly a constructive proof and give an  algorithm to construct N\'eron Desingularization in the easier case when $A$, $A'$ have the same residue field as in \cite[Theorem 10]{P}. We do not implement our algorithm. The implementation done in \cite{AP} gives a hard General N\'eron Desingularization because it is constructed in a general case, that is when $A$, $A'$ are Noetherian local domains of one dimension and with possible different residue fields and so involving some extra equations and new variables.
 In the last sections we present an algorithm to construct the easy case from \cite{KK} of the General  N\'eron Desingularization.

We owe thanks to the Referee who corrected our algorithms.

\section{Preliminaries on N\'eron Desingularization}

Let $A\subset A'$ be the unramified extension of DVR as above and $B$  a finite type sub-$A$-algebra of $A'$, let us say $B=A[y]$ for some elements $y=(y_1,\ldots,y_N) $ of $A'$. Let $\sigma :A[Y]\to A'$ be the $A$-morphism  $Y\to y$, $I=\mbox{Ker}\ \sigma$  and $f=(f_1,\ldots,f_r) $ be a regular system of parameters of the regular local ring $A[Y]_I$.
Thus $r=\mbox{height}\  I$. By separability of Fr $A\subset$ Fr $B$ we see that the Jacobian matrix $(\partial f/\partial Y)$ has a $r\times r$-minor $M$ such that $M(y)\not =0$.

 The minimum valuation $l(B)$ of the values  of the $r\times r$-minors of the Jacobian matrix $(\partial f_i/\partial Y_j)$ in $y$ is an invariant of $B.$ Indeed, tensorizing with $A'$ the module of differentials $\Omega_{B/A}$ we get the following exact sequence
 $$A'\otimes_B I/I^2\xrightarrow{\phi} \sum_{i\in [N]} A'd Y_i\to A'\otimes_B\Omega_{B/A}\to 0,$$
 where  $[N]=\{1,\ldots,N\}$ and $\phi$ is given by $g\to \sum_i (\partial g/\partial Y_i)(y)d Y_i$. Applying
  the Invariant Factor Theorem we see that the finite type $A'$-module  $A'\otimes_B\Omega_{B/A}$ is isomorphic with
$(\oplus_{i=1}^t A'/(x^{a_i}))\oplus A'^k$ for some $t,k,a_i\in \mathbb N$. Thus we may suppose that  Im $\phi$ has the diagonal form, where the only non zero elements are $(x^{a_i})$, that is $\sum_{i\in [t]} a_i$ is  the minimum valuation of the values  of the $t\times t$-minors of the Jacobian matrix $(\partial f_i/\partial Y_j)$ in $y$. It follows that $t=r$ and $l(B)=\sum_{i\in [r]} a_i$.

If $l(B)=0$ then we may choose $f_1,\ldots,f_r$ such that the Jacobian matrix $(\partial f/\partial Y)$ has an $r\times r$-minor $M$ such that $M(y)$ is invertible in $A'$, that is $M\not \in  q=\sigma^{-1}(xA')$. By Jacobian Criterion \cite[Theorem 30.4]{Mat} we see that $B_{xA'\cap B}\cong (A[Y]/I)_q$ is regular.

\begin{theorem} [N\'eron] \label{n} If $A\subset A'$ induces separable field extensions on fraction and residue fields then there exists a  sub-$A$-algebra of finite type $C$ of $A'$ containing $B$  such that $C_{xA'\cap C}$ is a regular local ring.
\end{theorem}
\begin{proof} If $B_{xA'\cap B}$ is regular then  we may take for $C$ a localization of $B$.  Suppose that $B_{xA'\cap B}$ is not regular. Then $f_1,\ldots,f_r$ cannot be part of a regular system of parameters of $A[Y]_q$, and so $f_1,\ldots, f_r$ induce a linearly dependent system of elements $\bar{f}$ in $qA[Y]_q/q^2A[Y]_q$. Also note that $l(B)>0$.
 Assume that ${\bar f}_1,\ldots, {\bar f}_e$, $e<r$, induce modulo $q^2$ a maximal linearly independent subsystem of $\bar f$ in  $qA[Y]_q/q^2A[Y]_q$.
 Then we may complete  ${\bar f}_1,\ldots, {\bar f}_e$ with  ${\bar h}_1,\ldots, {\bar h}_s$ up to a regular system of parameters of the regular ring
 $A/(x)[Y]_q$, thus $e+s=\mbox{height}\ q>\mbox{height}\ I=r$.   Since the field extension $A/(x)\to Q(A[Y]/q)$ is separable we see that the Jacobian matrix
$((\partial f_i/\partial Y_j)|(\partial h_k/\partial Y_j))_{i\in [e],k\in [s],j\in [N]}$ has rank $e+s$ modulo $q$. Note that ${\bar f}_i$ is linearly
 dependent on ${\bar f}_1,\ldots, {\bar f}_e$  modulo $q^2$  for $e<i\leq r$ and   there exist $E_i, L_{ic}\in A[Y]$, $E_i\not\in q$ such that
$$E_if_i-\sum_{c=1}^e L_{ic}f_c\in q^2.$$
Moreover we may choose $E_i, L_{ic}$ such that
$$E_if_i-\sum_{c=1}^e L_{ic}f_c\in (x,h_1,\ldots,h_s)^2,$$
for some $h_k\in A[Y]$ lifting ${\bar h}_k$. Set $g_k=xT_k-h_k\in A[Y,T]$, $k\in [s]$,  $T=(T_1,\ldots,T_s)$. Since $(x,h_1,\ldots,h_s)=(x,g_1,\ldots,g_s)$ we get
$$E_if_i-\sum_{c=1}^e L_{ic}f_c=\sum_{k,k'=1}^s S_{ikk'}g_kg_{k'}+x\sum_{k=1}^s F_{ik}g_k+x^2R_i,$$
for some $R_i,S_{ikk'}, F_{ik}\in A[Y,T]$. By construction, $h(y)\equiv 0$ modulo $x$, that is there exists $t\in A'^s$ such that $h(y)=xt$ and so $g_k(y,t)=0$. It follows that $R_i(y,t)=0$. Taking derivations above we get
$$E_i(y)(\frac{{\partial f_i}}{{\partial Y_j}})(y)-\sum_{c=1}^e L_{ic}(y)(\frac{\partial f_c}{\partial Y_j})(y)=
x\sum_{k=1}^s F_{ik}(y,t)(\frac{\partial g_k}{\partial Y_j})(y,t)+x^2(\frac{\partial R_i}{\partial Y_j})(y,t),$$
$$0=x\sum_{k=1}^s F_{ik}(y,t)(\frac{\partial g_k}{\partial T_{k'}})(y,t)+x^2(\frac{\partial R_i}{\partial T_{k'}})(y,t),$$
for $e<i\leq r$.
Note that the Jacobian matrix $J$ associated to the $(r+s)$-system \\
$(f_1,\ldots,f_{e},E_{e+1}f_{e+1},\ldots,E_rf_r,g_1,\ldots,g_s)$ has a minor $M$ whose value  in $(y,t)$ has the valuation $\leq l(B)+s$, because $E_i(y)\not \in xR'$ and the Jacobian matrix  $(\partial g_k/\partial T_{k'})$ is $x\mbox{Id}_s$, $\mbox{Id}_s$ being the identity matrix.

Moreover the valuation of $M(y,t)$ can be $\leq l(B)+r-e$.  Indeed, by elementary transformations on the first $r$-lines and the first $N$-columns of $J(y,t)$ we get a matrix $G$ with a diagonal $r\times N$-block on the above lines and rows. Since $(\partial (f_i,h_k)/\partial Y_j)$, $i\in [e],k\in [s]$, $j\in [N]$ has rank $e+s$ modulo $q$, we may choose an invertible $(s-r+e)\times (s-r+e)$-minor $U$ of the block $P$ of $(\partial h_k/\partial Y_j)$ given on the columns $r+1,\ldots, e+s$ after renumbering $Y$. Assume that $U$ is given on the lines $r+j_1,\ldots,r+j_{s-r+e}$ and let $\{u_1,\ldots,u_{r-e}\}=[s]\setminus \{j_1,\ldots,
j_{s-r+e}\}$.
 $G$ has the following form

 $$\left(
     \begin{array}{ccccc}
       \mbox{Id}_e & 0 & 0 & 0 & 0 \\
       0 & D & 0 & 0 & 0 \\
       \square & \square & P & \square & x\mbox{Id}_s \\
     \end{array}
   \right)
 $$
where $D$ is a diagonal matrix with $x^{a_{e+1}},\ldots,x^{a_{r}}$ on the diagonal and $l(B)=\sum_{i=1}^{r-e} a_i$.

   Then the $(r+s)\times (r+s)$-minor of $G$ given on the columns \\
$1,\ldots,r,r+1,\ldots, e+s, N+u_1,\ldots,N+u_{r-s}$ has the form $x^{l(B)+r-e}\det U$, which  has the
valuation $l(B)+r-e$.

On the other hand, adding   to the block of $J$ given by the rows  $e+1,\dots,r$ the first $e$ rows multiplied on left in order with $-L_{e+1,c}(y,t),\ldots, -L_{r,c}(y,t)$ and the last $s$ rows multiplied  on left in order with $-x F_{e+1,k}(y,t),\ldots,-xF_{r,k}(y,t)$ we get the Jacobian matrix associated to
$(f_1,\ldots,f_{e},x^2R_{e+1},\ldots,x^2R_r,g_1,\ldots,g_s)$. Thus we found a system of polynomials
$(f_1,\ldots,f_{e},R_{e+1},\dots,R_r,g_1,\ldots,g_s)$ in the kernel $Q$ of the $A$-morphism $A[Y,T]\to A'$ such that its Jacobian matrix has in $(y,t)$
a maximal minor of valuation $l(B)-r+e<l(B)$. Note that $C=B[t]$ has the same dimension with $B$ having the same fraction field and so $\mbox{height}\
 Q=N+s-\dim C= N+s-\dim B=r+s$. It follows that $l(C)<l(B)$. Using this construction by recurrence we arrive finally to a $C$ with $l(C)=0$ which is enough.
\hfill\ \end{proof}

\begin{remark} The above proof is not constructive since it is based on the  induction on $l(B)$. This is the reason to recall a constructive proof in the next section.
\end{remark}

\section{A Constructive N\'eron Desingularization}

A ring morphism $u:A\to A'$ of Noetherian rings has  {\em regular fibers} if for all prime ideals $P\in \mbox{Spec}\  A$ the ring $A'/PA'$ is a regular  ring, i.e. its localizations are regular local rings. It has {\em geometrically regular fibers}  if for all prime ideals $P\in \mbox{Spec}\  A$ and all finite field extensions $K$ of the fraction field of $A/P$ the ring  $K\otimes_{A/P} A'/PA'$ is regular. If $A\supset {\mathbb Q}$ then regular fibers of $u$ are geometrically regular. We call $u$ {\em regular} if it is flat and its fibers are geometrically regular. A regular morphism is {\em smooth} if it is finitely presented and it is {\em essentially smooth} if it is a localization of a finitely presented morphism.

The N\'eron Desingularization presented in the first section is a kind of
{\em smoothification} with respect to the inclusion $B\to A'$. Now we want
to find a smoothification with respect to an $A$-morphism $v:B\to A'$ in the case when $A=k[x]_{(x)}$, $A'=k[[ x ]]$, $k$ being a field and $x$ a variable.   Since we cannot give to a computer the whole informations concerning the coefficients of the formal power series $v(Y_j)$, it is better to work from the beginning with $A$-morphisms
$v:B\to A'/x^mA'$ for some $m\gg 0$. This is done in \cite[Theorem 10]{P}. Next we recall in sketch this construction since we need it for the algorithm. In \cite{AP}, \cite{PP} such algorithms are presented and even implemented for the so-called the General N\'eron Desingularization but in our case the things are simpler.

If $f=(f_1,\ldots,f_r)$, $r\leq n$ is a system of polynomials from $I$ then we consider an $r\times r$-minor $M$ of the Jacobian matrix $(\partial f_i/\partial Y_j)$.  Let  $c\in \mathbb N$. Suppose that there exist an $A$-morphism $v:B\to A'/(x^{2c+1})$ and $L\in ((f):I)$ such that
$A' v(LM)= (x)^c/(x)^{2c}$, where for simplicity we write $v(LM)$ instead $v(LM+I)$. We may assume that $M=\det((\partial f_i/\partial Y_j)_{i,j\in [r]})$.

\begin{theorem}[Popescu \cite{P}]\label{m} There exists a $B$-algebra $C$ which is smooth over $A$ such that every $A$-morphism $v':B\to A'$ with $v'\equiv v \ \mbox{modulo}\ x^{2c+1}$ (that is $v'(Y)\equiv v(Y) \ \mbox{modulo}\
 x^{2c+1}$) factors through $C$.
\end{theorem}
\begin{proof}

Since $A/(x^{2c+1})\cong A'/(x^{2c+1})$, $y'\in A^n$ can be choose such that \linebreak $v(Y)=y' + (x^{2c+1})$.
Set $P=LM$ and $d=P(y')$. We have $dA=x^cA$.

Let  $H$ be the $n\times n$-matrix obtained adding down to $(\partial f/\partial Y)$ as a border the block $(0|\mbox{Id}_{n-r})$. Let $G'$ be the adjoint matrix of $H$ and $G=LG'$. We have
$GH=HG=LM \mbox{Id}_n=P\mbox{Id}_n$
and so
$d\mbox{Id}_n=P(y')\mbox{Id}_n=G(y')H(y').$

 Let
 $h=Y-y'-dG(y')T,$
 where  $T=(T_1,\ldots,T_n)$ are new variables.  Since
$Y-y'\equiv dG(y')T\ \mbox{modulo}\ h$
and
$f(Y)-f(y')\equiv \sum_j\partial f/\partial Y_j(y') (Y_j-y'_j)$
modulo higher order terms in $Y_j-y'_j$, by Taylor's formula we see that we have
$$f(Y)-f(y')\equiv  \sum_jd\partial f/\partial Y_j(y') G_j(y')T_j+d^2Q=
dP(y')T+d^2Q=d^2(T+Q)$$
$\mbox{modulo}\ h,$
 where $Q\in T^2 A[T]^r$. This is because $(\partial f/\partial Y)G=(P\mbox{Id}_r|0)$.  We have $f(y')=d^2a$ for some $a\in xA^r$.
 Set $g_i=a_i+T_i+Q_i$, $i\in [r]$ and  $E=A[Y,T]/(I,g,h)$.
 Then there exists $s,s'$ in $1+(T)$ as it is proved in \cite[Theorem 10]{P} such that $C:=E_{ss'}\cong (A[T]/(g))_{ss'}$ is smooth . Moreover, $v'$ factors through  $C $.
\hfill\ \end{proof}
\vskip 0.3 cm
\section{N\'eron Desingularization Algorithm}\label{algo1}
\vskip 0.3 cm
{\bf Neron-Desingularization\_Dim1}
\vskip 0.3 cm
Input: $N\in \mathbb{Z}_{>0}$ a bound. $ A=k[x]_{(x)}$, $k$ being a field, $A'=k[[ x ]]$,
$B=A[Y]/I,\linebreak I=(f_1,\ldots, f_q), f_i\in k[x,Y], Y=(Y_1,\ldots,Y_n)$,  $v:B\to A'$ an $A$-morphism and  $y' \in {k[x]}^n$ approximations mod $(x)^N$ of  $v(Y)$.

Output: $(C,\pi)$ given by the following data: $C=(A[Z]/(L))_{hM}$ standard smooth, $Z=(Z_1,\ldots,Z_p)$, $L=(b_1,\ldots,b_{q'})\subset k[x,Z]$, $h\in k[x,Z]$, $M$ a $q\times q$-minor of $(\partial b_i/\partial Z_j)$, $\pi:B\to C$ an $A$-morphism given by $\pi(Y_1),\ldots,\pi(Y_n)$ factorizing $v$,  or the message ``$y'$, $N$ are not well chosen''.

\begin{enumerate}

  \item Compute $f:= (f_1,\ldots,f_r)$ in $I$  such that $v(((f):I)\Delta_f)\not\subset (x)^N $

  \item Reorder the variables $Y$ such that for $M:=\det(\partial f_i/\partial Y_j)_{i,j\in [r]})$ and a suitable $L\in (f):I $ we have $v(LM)\not \in (x)^N$

  \item $P:=LM$; $d:=P(y')$; $c:=ord(d)$

  \item If $2c+1 > N$, return ``$y'$, $N$ are not well chosen''
  \item Complete $(\partial f_i/\partial Y_j)_{i \leq r}$  with $(0| (\mbox{Id}_{n-r}))$ in order to obtain a square matrix $H$
  \item Compute $G'$ the adjoint matrix of $H$ and $G:=LG'$
  \item $h=Y-y'-d G(y')T,\  T=(T_1,\ldots, T_n)$
  \item Compute $Q\in T^2 A[T]^r$ such that \\ $f(Y)-f(y')= \sum_j d(\partial f/\partial Y_j)(y') G_j(y')T+d^{2}Q$
  \item Compute $a\in xA^r$ such that $f(y')=d^{2}a$
  \item For $i=1$ to $r$, $g_i=a_i+T_i+Q_i$
  \item $E:=A[Y,T]/(I,g,h)$
  \item Compute $s$ the $r\times r$ minor defined by the first $r$ columns of $(\partial g/\partial T)$
  \item Compute $s'$ such that $P(y'+ d G(y')T)=ds'$
  \item return $C:=E_{ss'}$ and the canonical map $\pi:B\to C$.
\end{enumerate}

\begin{example}\label{example1}{\em
 Let $N=10$,
$A={\mathbb Q}[x]_{(x)}$, $A'={\mathbb Q}[[ x ]]$,
$B=A[Y_1,Y_2,Y_3,Y_4]/(f_1, f_2),$ $f_1=Y_1^3-Y_2^2, f_2=3{Y_1}^2Y_3-2Y_2Y_4$,
$r=2$,
$v:B\rightarrow {\mathbb Q}[[ x ]], v(Y_1)=x^2u_1,$ $v(Y_2)=x^3u_2,v(Y_3)=v_1, v(Y_4)=xv_2$ where $u_1,v_1$ are two formal power series from ${\mathbb Q}[[ x ]]$ which are algebraically independent over ${\mathbb Q}(x)$, $u_1(0)=v_1(0)=1$ and $u_2,v_2$ are defined as it follows. By the Implicit Function
Theorem there exists $u_2\in \mathbb Q[[ x ]]$ such that $u_2^2=u_1^3$. Set $v_2=(3/2)u_1^2v_1u_2^{-1}$ and
$L=1.$ Now we follow the steps of algorithm

We get the following minors: $M_1=6Y_1^2Y_4+12Y_1Y_2Y_3, M_2=9{Y_1}^4, M_3=6{Y_1}^2Y_2,$ $M_4=4{Y_2}^2, M_5=0.$ Note that $v(M_i)\in (x)$ for all the minors $M_i$. We choose $M_i$ for which the valuation of $M_i(y)$ is minimum, in this case $M_1$ which is maped by $v$ in $x^5(6u_1^2v_2+12u_1u_2v_1)$. Then $c=5$ and $2c+1>10.$

Output: $y'$, $N$  are not well chosen.}
\end{example}

\begin{example}\label{example2}{\em
Let $A={\mathbb Q}[x]_{(x)}$, $A'={\mathbb Q}[[ x ]]$,
$B=A[Y_1,Y_2,Y_3,Y_4]/(f_1, f_2)$, \linebreak
$f_1=Y_1^3-Y_2^2, f_2=3{Y_1}^2Y_3-2Y_2Y_4.$ The only difference from Example \ref{example1} consists in the map $v:B\rightarrow {\mathbb Q}[[ x ]]$, $v(Y_1)=u_1,v(Y_2)=u_2$,
$v(Y_3)=v_1$, $v(Y_4)=v_2$, where $u_1,u_2,v_1,v_2,L$ are
as in Example \ref{example1}.

We now follow the steps of algorithm.
We get the following minors: $M_1=6Y_1^2Y_4+12Y_1Y_2Y_3, M_2=9{Y_1}^4,M_3=6{Y_1}^2Y_2, M_4=4{Y_2}^2, M_5=0.$ Note that $v(LM_i)\notin (x)$ for all the
minors $M_i$. We take $M=4Y_2^2$, which is mapped by $v$ in $4u_2^2 $ and thus $v(M)\notin (x)$  is invertible, that is $c=0$. So $C=B_{M}\cong {(A[Y_1,Y_2]/(Y_1^3-Y_2^2))}_{4Y_2^2}$ is smooth over $A.$}
\end{example}

\begin{example}\label{example3}{\em
Let $A={\mathbb Q}[x]_{(x)}$, $A'={\mathbb Q}[[ x ]]$, $B=A[x^2u_1,x^3u_2,v_1,xv_2]$, where $u_1$,$u_2$,$v_1$,$v_2$ are as in Example \ref{example1}.

Let $\phi:A[Y_{1},Y_{2},Y_{3},Y_{4}]\to B$ be given by $Y_{1}\to x^2u_1$, $Y_{2}\to x^3u_2$, $Y_{3}\to v_1$, $Y_{4}\to xv_2$.
Then $\mbox{Ker}\ \phi$ contains $g_1=Y_{1}^3-Y_{2}^2$, $g_2=3Y_{1}^2Y_{3}-2Y_{2}Y_{4}$. Set $a_1=27Y_{2}Y_{3}^3-8Y_4^3$, $a_2=9Y_{1}Y_{3}^2-4Y_{4}^2$,
$a_3=2Y_{1}Y_{4}-3Y_{2}Y_{3}$. The minimal prime ideals of $(g_1,g_2)$ are $P_1=(g_1,g_2,a_1,a_2,a_3)$, $P_2=(Y_{1},Y_{2})$, both ideals having the
height $\geq 2$ because the great common divisor of $g_1,g_2$ is one. Note that  $\mbox{height}\ \mbox{Ker}\ \phi=2$ because $\dim(A[Y]/\mbox{Ker}\ \phi)=\dim(B)=2$, $u_1,v_1$ being algebraically independent over ${\mathbb Q}[x]$.

We have $\mbox{Ker}\ \phi\supset P_1$ since $\mbox{Ker}\ \phi\supset (g_1,g_2)$ and $Y_{1}\notin \mbox{Ker}\ \phi$. In fact
$\mbox{Ker}\ \phi= P_1 $ because $\mbox{height} P_1\geq \mbox{height}\ \mbox{Ker}\ \phi$ as above.
 Note that $Y_{2}^2 a_2 \in (g_1,g_2)$, $Y_{2}a_3 \in (g_1,g_2)$, $Y_{2}^3a_1 \in (g_1,g_2)$, which implies $Y_{2}^3 \mbox{Ker}\ \phi \subset (g_1,g_2)$ and so $Y_{2}^3 \in ((g_1,g_2):\mbox{Ker}\ \phi)$. Here $L=Y_2^3.$

The Jacobian matrix ($\partial g_k/\partial Y_{i})$ contains a $2\times 2$-minor $M=4Y_{2}^2$ and it follows that $Y_2\in H_{B/A}$. So $\phi(M((g_1,g_2):\mbox{Ker}\ \phi))$ contains $\phi(Y_{2}^5)=4x^{15}u_2^5$. Taking $c=15$ we see that $(\phi(M((g_1,g_2):\mbox{Ker}\ \phi)))=(x^c)$.

Take $E=A[u_1,u_2,v_1,v_2]$. The kernel of the map $\psi:A[U_1,U_2,V_1,V_2]\to E$ given by $U_i\to u_i$, $V_i\to v_i$, contains
the polynomials $h_1=U_1^3-U_2^2$, $h_2=3{U_1}^2V_1-2U_2V_2$ and we set $b_1=27U_{2}V_{1}^3-8V_2^3$, $b_2=9U_{1}V_{1}^2-4V_{2}^2$,
$b_3=2U_{1}V_{2}-3U_{2}V_{1}$. The minimal prime ideals of $(h_1,h_2)$ are $Q_1=(h_1,h_2,b_1,b_2,b_3)$, $Q_2=(U_{1},U_{2})$.
As above we see that $\mbox{Ker}\ \psi=Q_1$ since $U_1\notin \mbox{Ker}\ \psi$ and so  $\psi$ induces
 the isomorphism $A[U_{1},U_{2},V_{1},V_{2}]/(h_1,h_2,b_1,b_2,b_3)\cong E$. Note that $U_{2}^2 b_2 \in (h_1,h_2)$, $U_{2} b_3 \in (h_1,h_2)$, $U_{2}^3 b_1 \in (h_1,h_2)$ which implies $U_{2}^3 \mbox{Ker}\ \psi \subset (h_1,h_2)$ and so $U_{2}^3 \in ((h_1,h_2):\mbox{Ker}\ \psi)$.

Now the Jacobian matrix $\partial h/\partial (U_i,V_i)$
contains a minor $M'=4U_2^2$ and so $ U_{2} \in H_{E/A}$. Note that $\psi(M'((h_1,h_2):\mbox{Ker}\ \phi))$ contains $\psi(4U_{2}^5)$ which is mapped by $\psi$ in $4u_2^5\not \in (x)$. Then $v$ factors through the smooth $A$-algebra $$C=(A[U_1,U_2,V_1,V_2]/(h_1,h_2))_{U_2}\cong (A[U_1,U_2,V_1]/(h_1))_{U_2}$$
 because $v$ is the composite map
$$B\cong A[Y_{1},Y_{2},Y_{3},Y_{4}]/(g_1,g_2,a_1,a_2,a_3)\xrightarrow{\rho} E\to C \to {\mathbb Q}[[ x ]],$$
 where $\rho $ is given by $Y_{1}\to x^2u_1$,$Y_{2}\to x^3u_2$, $Y_3\to v_1$, $Y_4\to xv_2$.}
\end{example}

\begin{remark}
{\em In Example \ref{example3}, we gave a General N\'eron Desingularization using no algorithm. Applying our algorithm we will get a more complicated General N\'eron Desingularization.
Example \ref{example4}  illustrates the whole construction of the proof from Theorem \ref{m}.}
\end{remark}
\begin{example}\label{example4}{\em
Let $N=7$ be a bound.
Let $A={\mathbb Q}[x]_{(x)}$, $A'={\mathbb Q}[[ x ]]$, $B=A[Y_1,Y_2]/(f)$,  $f= Y_1^3-Y_2^2$, and $v:B\to A'$ be a morphism given by $Y_1\to x^2u_1$, $Y_2\to x^3u_2$, where $u_1,u_2,$ are as in Example \ref{example1}. Let $u'_1=\sum_{i=0}^{7}\frac{x^i}{i!}$ and compute $u'_2$ as in the previous example. Let $y'$ be given by $u'_1$ and $u'_2$.
Now we follow the steps of algorithm:

\begin{enumerate}

  \item $f=f $
  \item $Y=(Y_1,Y_2)$. Actually the order taken  in the  algorithm was $(Y_2,Y_1)$.  Among the minors $M_1=3Y_1^2, M_2=2Y_2$, we choose $M=2Y_2$; $L=1$ and $v(LM)=2x^3u'_2 \notin (x)^7 $
  \item $P=2Y_2$, $P(y')=2x^3u'_2$, to avoid complexity we take $d=x^3$; $c=3$.
  \item $2c+1=7=N $
  \item $H=\left(
          \begin{array}{cc}
            3Y_1^2 & -2Y_2 \\
            1 & 0 \\
          \end{array}
        \right)$
  \item $G=LG'=\left(
          \begin{array}{cc}
            0 & 2Y_2 \\
            -1 & 3Y_1^2 \\
          \end{array}
        \right)$
  \item $h_1=Y_1-x^2u'_1-2x^6u'_2T_2$,\\
       $h_2=Y_2-x^3u'_2+x^3T_1-3x^7u'^2_1T_2$\\
  \item $Q=-2T_1^2+12x^4u'^2_1T_1T_2 +(24 x^8u'^2_2u'_1-18 x^8 u'^4_1)T_2^2+48x^{12}u'^3_2T_2^3$\\

  \item $f(y')=x^6\cdot a$ where $a=x\alpha \in xA$\\
  \item $g=x\alpha +2u'_2T_1-2T_1^2+12x^4u'^2_1T_1T_2 +(24 x^8u'^2_2u'_1-18 x^8 u'^4_1)T_2^2+48x^{12}u'^3_2T_2^3 $\\
  \item $E=A[Y_1,Y_2,T_1,T_2]/(f,g,h_1,h_2)$\\
  \item $s=2u'_2-4T_1+12x^4u'^2_1T_2$\\
  \item $s'=2u'_2-T_1+3x^4u'^2_1T_2$\\
  \item $C=E_{ss'}\cong (A[T]/g)_{ss'} $
\end{enumerate}

Set $b:=Y-h\in A[T]^2$. Then the above isomorphism is induced by the $A$-morphism
$A[Y,T]\to A[T]$, $Y\to b$.}

\end{example}

We can also compute the Example \ref{example4} in \sing \ using \verb"GND.lib" given in \cite{AP} but the result is harder.

\begin{example}\label{example5}{\em
Let $N=31$ be a bound.
Let $A={\mathbb Q}[x]_{(x)}$, $A'={\mathbb Q}[[ x ]]$,\linebreak $B=A[x^2u_1,x^3u_2,v_1,xv_2]$, where $u_1,u_2,v_1,v_2$ are as in Example \ref{example1}. Suppose that $u_1= \sum_{i=0}^{\infty}\frac{x^i}{i!}, v_1=\sum_{i=0}^{\infty}{i!}{x^i}$. Suppose also that $u'_1=\sum_{i=0}^{31}\frac{x^i}{i!} $ and $v'_1=\sum_{i=0}^{31}{i!}{x^i}$ and we will get $u'_2$ and $v'_2$ according to the relations in Example \ref{example1}. Let $y'$ be given by $u'_1$,
$u'_2$, $v'_1$, $v'_2$.
Let $v: B\to A'$ be the inclusion, that is in fact the map $B\cong  A[Y_{1},Y_{2},Y_{3},Y_{4}]/(f_1,f_2,f_3,f_4,f_5)\to A'$ given by $Y_1\to x^2u_1$, $Y_2\to x^3u_2$, $Y_3\to v_1$, $Y_4\to xv_2$ where $f_1=Y_{1}^3-Y_{2}^2$, $f_2=3Y_{1}^2Y_{3}-2Y_{2}Y_{4}$, $f_3=27Y_{2}Y_{3}^3-8Y_4^3$, $f_4=9Y_{1}Y_{3}^2-4Y_{4}^2$, $f_5=2Y_{1}Y_{4}-3Y_{2}Y_{3}$ same as in Example \ref{example3}.
Now we follow the steps of algorithm:
\begin{enumerate}
  \item $f=(f_1,f_2) $.
  \item $Y=(Y_1,Y_2,Y_3,Y_4)$. Among the minors $M_1=6Y_1^2Y_4+12Y_1Y_2Y_3,\\
   M_2=9{Y_1}^4, M_3=6{Y_1}^2Y_2, M_4=4{Y_2}^2, M_5=0$ we choose $M=4Y_2^2.$ $L=Y_2^3$; and $v(LM)=4x^{15}{u'_2}^5 \notin (x)^{31}.$
  \item $P=4Y_2^5$, $P(y')=4x^{15}u'^5_2$, $d=x^{15}$ and $c=15.$
  \item $2c+1=31.$
  \item
        $H= \left(
          \begin{array}{cccc}
            3Y_1^2 & 0 & 0 & -2Y_2 \\
            6Y_1Y_3& 3Y_1^2 & -2Y_2 & -2Y_4 \\
            1 & 0 & 0 & 0 \\
            0 & 1 & 0 & 0 \\
          \end{array}
        \right)$
  \item $G=LG'= \left(
          \begin{array}{cccc}
            0 & 0 & -4Y_2^5 & 0 \\
            0 & 0 & 0 & -4Y_2^5 \\
            -2Y_2^3Y_4 & 2Y_2^4 & -12Y_1Y_2^4Y_3+6Y_1^2Y_2^3Y_4 & -6Y_1^2Y_2^4 \\
            2Y_2^4 & 0 & -6Y_1^2Y_2^4 & 0 \\
          \end{array}
          \right)$

\end{enumerate}
We stop here with the algorithm since the computations of $h,g, s,s'$ are  difficult already.}
\end{example}

\section{A Constructive  General N\'eron Desingularization in a special case}

Let $(A,\mathfrak{m})$ be a local Artinian ring, $(A',\mathfrak{m}')$ a Noetherian complete local ring of dimension one such that $k=A/\mathfrak{m}\cong A'/\mathfrak{m}'$, and $u:A\rightarrow A'$ be a regular morphism. Suppose that  $k\subset A$. Then $\bar {A}'=A'/\mathfrak{m}A'$ is a discrete valuation ring (shortly a DVR). Choose $x\in A'$ such that its class modulo $\mathfrak{m}A'$ is a local parameter of $\bar A'$, that is, it generates  $\mathfrak{m}'\bar A'$.
Let $B=A[Y]/I$, $Y=(Y_1,\ldots,Y_n)$.
\begin{theorem}[Khalid-Kosar \cite{KK}] \label{bulletin}
Then any morphism $v: B\rightarrow A'$ factors through a smooth $A$-algebra $C$.
\end{theorem}
\begin{proof} Here we recall in sketch the proof from \cite{KK} because we
need it in the next algorithm.
 Let $A_1=A[x]_{(x)}$ and $u_1$ be the inclusion $A_1\subset A'$. Then $u_1$ is a regular morphism.  Let $B_1=A_1\otimes_A B$ and $v_1:B_1\rightarrow A'$ be
 the composite map $a_1\otimes b\mapsto u_1(a_1)\cdot v(b)$.

There exists a certain $s$ such that $\mathfrak{m}^s=0$ because $A$ is an Artinian local ring and so $A$ has the form $A=k[T]/{\mathfrak a}$,
 $T=(T_1,\ldots,T_m)$, and the maximal ideal of $A$ is generated by $T$. Then for all $i\in [m]=\{1,\ldots, m\}$, ${T_i}^s\in {\mathfrak a}$ and
 $A'=k[[ x ]][T]/(\mathfrak{a}) \cong A \otimes_k k[[ x ]]$. Note that
$v(Y_i)=\hat{y}_i$ has the form $\sum_{\alpha \in \mathbb{N}^m, |\alpha|<s} y_{i\alpha} T^\alpha$,  $T^\alpha={T_1}^{\alpha_1}\cdots {T_m}^{\alpha_m},\ |\alpha|=\alpha_1+\cdots+\alpha_m$ and $y_{i\alpha}\in {\bar A}'=k[[ x ]]$.

Set ${\bar B}_1=\bar{A}_1[(y_{i\alpha})_{\alpha}]\subset k[[ x ]]$ and  let ${\bar v}_1$ be this inclusion. Then $v$ factors through $B_1=A\otimes_k {\bar B}_1\subset A'$, that is $v$ is the composite map $B\xrightarrow{q} B_1\xrightarrow{A\otimes_k{\bar v}_1} A'$,
 where $q$ is defined by $Y_i\to \sum_{\alpha} T^{\alpha}\otimes y_{i\alpha}$. Applying  Theorem  \ref{m} to the case ${\bar A}_1=k[x]_{(x)}$, $\bar A'$, ${\bar B}_1$ and ${\bar v}_1={\bar A}_1\otimes_{A_1}v_1$
 we see that
${\bar v}_1$ factors through a smooth $k$-algebra ${\bar C}$.  Then $A\otimes_k{\bar v}_1$ factors through $A\otimes_k {\bar C}$. It follows that $v$   factors through a smooth $A$-algebra $C$ (see  e.g. \cite[Lemma 1]{KK}). \hfill\
\end{proof}

\section{A special Algorithm}\label{algo2}

In our next algorithm we will use the N\'eron Desingularization algorithm given in Section \ref{algo1}.
\vskip 0.3 cm
{\bf Special-Neron-Desingularization}
\vskip 0.3 cm

Input: $N\in \mathbb{Z}_{>0}$ a bound \\
$ A=k[T]/(a),$ $a=(a_1,\ldots,a_e)$  $T=(T_1,\ldots,T_m), {T_i}^s\in (a)$, $A'=k[[ x ]][T]/(a),$
$B=A[Y]/I$, $I=(g_1,\ldots, g_l), g_i\in k[T,Y], Y=(Y_1,\ldots,Y_n)$, integers $q$, $\alpha$.
$v:B\to A'$ an $A-$morphism given by $v(Y_i)=\hat{y}_i= \sum_{\alpha \in \mathbb{N}^m, |\alpha|<s} y_{i\alpha}\cdot T^\alpha$,   $y_{i\alpha}\in k[[ x ]]$,
 ${\bar A}'=k[[ x ]]$.

Output: A N\'eron Desingularization $(C,\pi)$ of $v:B\rightarrow A'$ or the message ``the algorithm fails since the bound is too small''.

\begin{enumerate}
  \item $\bar{A}_1=k[x]_{(x)}$, $\bar{B}_1:=\bar{A}_1[({y}_{i\alpha})_\alpha]$, $\bar{v}_1$ is the inclusion $\bar{B}_1\subset \bar{A}'$
  \item Write $(\bar{C},\bar{\pi}):=$ Neron-Desingularization\_Dim1 for $N,\bar{A}_1, \bar{A'}, \bar{B}_1, \bar{v}_1$
  \item $\bar{C}:=E_{ss'}=(\bar{A_1}[(Y_{i\alpha})_{\alpha},T]/L)_{ss'}$ where $L=<(l_i)>$

  \item $\bar{C}:=\bar{A}_1\otimes_k \tilde{C}$ where $\tilde{C}=((k[x,(Y_{i\alpha})_{\alpha},T]/\tilde{L})_{\tilde{s}\tilde{s}'})_{\eta}$ and
  $\tilde{s}=s\cdot \eta$, $\tilde{s}'=s'\cdot \eta$, $\tilde{L}=<(\tilde{l_i})>,\tilde{l_i}=l_i\cdot\eta$, where
  $\eta \in k[x] \setminus (x)$
  \item $C:=A\otimes_k \tilde{C}$, $\pi$ is induced by $\bar{\pi}$
  \item return $(C,\pi)$.
\end{enumerate}

\begin{remark}{\em
Here we give two examples for the same rings. Example \ref{buletin1} gives a N\'eron Desingularization which comes from the direct computations, while Example \ref{buletin2} gives a smooth $A-$algebra $C$ which we get by following the algorithm of Section \ref{algo1}.}
\end{remark}

\begin{example}\label{buletin1}{\em

Let $A={\mathbb Q}[t]/(t^2)$, $A'={\mathbb Q}[[ x ]][t]/(t^2)$, $A_1={\mathbb Q}[x]_{(x)}[t]/(t^2)$,
$B=A[Y_1,Y_2]/(Y_1^3-Y_2^2)$ and $u_1,v_1$ two formal power series from ${\mathbb Q}[[ x ]]$
which are algebraically independent over ${\mathbb Q}(x)$ and $u_1(0)=v_1(0)=1$. By the Implicit Function
Theorem there exists $u_2\in \mathbb Q[[ x ]]$ such that $u_2^2=u_1^3$. Set $v_2=(3/2)xu_1^2v_1u_2^{-1}$,
${\hat y}_1=x^2u_1+tv_1$, ${\hat y}_2=x^3u_2+txv_2$. We have
$g({\hat y}_1,{\hat y_2})=x^6(u_1^3-u_2^2)+tx^4(3u_1^2v_1-2u_2v_2)=0$
and we may define $v:B\to A'$ by $Y\to ({\hat y}_1,{\hat y_2})$.

Take $\bar{B}_1= \bar{A}_1[x,x^2u_1,x^3u_2,v_1,xv_2].$ Then ${\bar v}={\mathbb Q}\otimes_Av: B/tB\to {\mathbb Q}[[ x ]]$
factors through $\bar{B}_1$. Now $\bar{B}_1=\bar{A}_1[(Y_{i\alpha})_{\alpha}]/J$ as in Example \ref{example3}.
Since $\bar{A}_1=k[x]_{(x)}$, $\bar{A'}=k[[ x ]]$, $\bar{B}_1=\bar{A}_1[(Y_{i\alpha})_{\alpha}]/J$ so applying the algorithm from Section \ref{algo1} for $\bar{A}_1, \bar{A'}, \bar{B}_1$ we get $\bar{C}=(\bar{A}_1[U_1,U_2,V_1]/(h_1))_{U_2}=\bar{A}_1\otimes_k \tilde{C}$ where $\tilde{C}= (k[U_1,U_2,V_1]/(h_1))_{U_2}$ and $U_1,U_2,V_1, h_1$ are as in Example \ref{example3}.
So $C=A\otimes_k \tilde{C}\cong (A[U_1,U_2,V_1]/(h_1))_{U_2}$.}

\end{example}

\begin{example}\label{buletin2}{\em
Considering everything as in Example \ref{buletin1} until we apply the algorithm from Section \ref{algo1}, we get $\bar{C}=E_{ss'}$ where
$E_{ss'}$ is the same as in example \ref{example4}. So $C=A\otimes_k \tilde{C}$ where $\tilde{C}$ can
be obtained as in Example \ref{buletin1}.}
\end{example}

\noindent\textbf{Acknowledgement}\textit{The first author  gratefully acknowledges the support from the ASSMS GC. University Lahore, for arranging her visit to Bucharest, Romania and she is also grateful to the Simion Stoilow Institute of the Mathematics of the Romanian Academy for inviting her. }

\smallskip\noindent
Received:


\address{Abdus Salam School of Mathematical Sciences,GC University, Lahore, Pakistan.}
\email{asmakhalid768@gmail.com}

\address{Adrian Popescu}
\email{adrian.mihail.popescu@gmail.com}

\address{Dorin Popescu, Simion Stoilow Institute of Mathematics of the Romanian Academy, Research unit 5,
University of Bucharest, P.O.Box 1-764, Bucharest 014700, Romania}
\email{dorin.popescu@imar.ro}


\begin{thebibliography}{9}

\bibitem{A} M.\ Artin, {\em Algebraic approximation of structures over complete local rings}, Publ. Math. IHES, {\bf 36} (1969), 23-58.
\bibitem{Sing} W.\ Decker, G.-M.\ Greuel, G.\ Pfister, H.\ Sch{\"o}nemann: \newblock {\sc Singular} {3-1-6} --- {A} computer algebra system for polynomial  computations.\newblock {http://www.singular.uni-kl.de} (2012).

\bibitem{KK} A.\ Khalid, Z.\ Kosar, {\em An Easy proof of the General
Neron Desingularization in dimension 1},
Bull. Math. Soc. Sci. Math. Roumanie,
{\bf 59 (107)}, 2016, 349-353.

\bibitem{KPP} A.\ Khalid, G. Pfister, D.\ Popescu, {\em An uniform  General Neron Desingularization in dimension one}, arXiv:AC/1612.03416.
\bibitem{Mat} H. Matsumura, {\em Commutative Ring Theory}, Cambridge Univ. Press, Cambridge, 1986.

\bibitem{N} A.\ N\'eron, {\em Modeles minimaux des varietes abeliennes sur les corps locaux et globaux}, Publ. Math.  IHES, {\bf 21}, 1964, 5-128.
\bibitem{PP0} G.\ Pfister, D.\ Popescu, {\em  Die strenge Approximationseigenschaft lokaler Ringe}, Inventiones Math.
{\bf 30}, (1975), 145-174.

\bibitem{PP} G. Pfister, D. Popescu, {\em Constructive General Neron Desingularization for one dimensional local rings}, Journal of Symbolic Computation, {\bf 80}, (2017), 570-580, arXiv:AC/1512.08435v1.

\bibitem{Pl} A.\ Ploski, {\em Note on a theorem of M. Artin}, Bull. Acad. Polon. des  Sci., t. XXII, {\bf 11} (1974), 1107-1110.


\bibitem{AP} A.\ Popescu, D.\ Popescu, {\em A method to compute the General Neron Desingularization in the frame of one dimensional local domains}, to appear in "Singularities and Computer Algebra - Festschrift for Gert-Martin Greuel
   On the Occasion of his 70th Birthday", Editors Wolfram Decker, Gerhard Pfister, Mathias Schulze,
 Springer Monograph., DOI 10.1007/978-3-319-28829-1, arXiv:AC/1508.05511.
\bibitem{P'} D.\ Popescu, {\em  A strong approximation theorem over discrete valuation rings}, Rev.
 Roum. Math. Pures et Appl., {\bf  20}, (1975), 659-692.


 \bibitem{P0} D.\ Popescu, {\em General Neron Desingularization and approximation}, Nagoya Math. J., {\bf 104} (1986), 85-115.
\bibitem{P1} D.\ Popescu, {\em Artin Approximation}, in "Handbook of Algebra", vol. 2, Ed. M. Hazewinkel, Elsevier, 2000, 321-355.
\bibitem{P2} D.\ Popescu,  {\em Lecture 3,  Neron Desingularization and its generalizations} in the Conference {\bf Algebraic versus Analytic Geometry} organized by H. Hauser in  Erwin Schr\"odinger Institut for Mathematical  Physics from Wien, November-December 2011\\ (see
 http://www.esi.ac.at/activities/events/2011/algebraic-versus-analytic-geometry).
 \bibitem{P}  D. Popescu, {\em Around General Neron Desingularization}, Journal of Algebra and Its Applications, {\bf 16}, No. 2 (2017), doi: 10.1142/S0219498817500724, arXiv:1504.06938.


\bibitem{S} R.\ Swan, {\em Neron-Popescu desingularization}, in "Algebra and Geometry", Ed. M. Kang, International Press, Cambridge, (1998), 135-192.




%
%
%
%
%
\end{thebibliography}
\end{document}